\documentclass[10pt]{amsart}

\usepackage{amsfonts, amssymb, amsmath, enumerate}

\numberwithin{equation}{section}

\DeclareMathOperator{\E}{\mathbb{E}}

\def \C {\mathbb{C}}
\def \N {\mathbb{N}}

\def \R {\mathbb{R}}

\def \E {\mathbb{E}}

\def \vol {{\rm vol}}

\def \etc {,\ldots,}

\newtheorem{theorem}{Theorem}[section]
\newtheorem{proposition}[theorem]{Proposition}
\newtheorem{corollary}[theorem]{Corollary}
\newtheorem{lemma}[theorem]{Lemma}

\theoremstyle{remark}

\begin{document}

\title{On hyperplane sections and projections in $l_p^n$}

\author{Hermann K\"onig (Kiel)}

\keywords{Volume, hyperplane sections, projections, $l_p^n$-ball, random variables}
\subjclass[2000]{Primary: 52A38, 52A40 Secondary: 46 B07, 60F05}

\begin{abstract}
For $2 < p < p_0 \simeq 26.265$, the hyperplane section of the $l_p^n$-unit ball $B_p^n$ perpendicular to $a^{(n)} = \frac 1 {\sqrt n} (1 \etc 1)$ for large $n$ has larger volume than the one orthogonal to $a^{(2)} = \frac 1 {\sqrt 2} (1,1,0 \etc 0)$, as shown by Oleszkiewicz. This is different from the case of $l_\infty^n$ considered by Ball. We give a quantitative estimate for which dimensions $n$ this happens, namely for $n > c (\frac 1 {p_0-p} + \frac 1 {p-2})$ for some absolute constant $c>0$. Correspondingly for projections of $B_q^n$ onto hyperplanes, Barthe and Naor showed that projections onto hyperplanes perpendicular to $a^{(n)}$ have smaller volume for large $n$ than onto the one orthogonal to $a^{(2)}$, if $\frac 4 3 < q < 2$, different from the case $q=1$. We show that this happens for all $n > 5 (\frac 1 {q-\frac 4 3} + \frac 1 {2-q})$.
\end{abstract}

\maketitle

\section{Introduction and main results}

In a well-known paper Ball \cite {B} proved that the hyperplane section of the $n$-cube perpendicular to $a^{(2)} = \frac 1 {\sqrt 2} (1,1,0 \etc 0) \in S^{n-1} \subset \R^n$ has maximal volume among all hyperplane sections. Earlier Hadwiger \cite{Ha} and Hensley \cite{He} had shown independently of one another that coordinate hyperplanes, e.g. orthogonal to $a^{(1)} = (1,0 \etc 0) \in S^{n-1}$, yield the minimal $(n-1)$-dimensional cubic sections. \\

Meyer and Pajor \cite{MP} found extremal sections of the $l_p^n$ balls $B_p^n$: They proved that the normalized volume of sections of $B_p^n$ by a fixed hyperplane is monotone increasing in $p$. This implies that coordinate planes provide the minimal sections for $2 \le p < \infty$, as for $p=\infty$, and the maximal sections for $1 \le p \le 2$. The minimal hyperplane sections of $B_1^n$ are those orthogonal to a main diagonal, e.g. $a^{(n)} = \frac 1 {\sqrt n}(1 \etc 1) \in S^{n-1}$, see also \cite{MP}. Koldobsky \cite{K} extended this to the full range $1\le p \le 2$. \\

This left open the case of the maximal hyperplane section of $B_p^n$ for $2 < p < \infty$. The situation there is more complicated, since then the maximal hyperplane may depend as well on $p$ as on the dimension $n$: Oleszkiewicz \cite{O} proved that Ball's result does not transfer to the balls $B_p^n$ if $2 < p < p_0 \simeq 26.265$: the intersection of the hyperplane perpendicular to $a^{(n)}$ has larger volume than the one orthogonal to $a^{(2)}$, for sufficiently large dimensions $n$. Oleszkiewicz' result is an asymptotic one, not determining dimensions $n$ for which this happens. We derive a quantitative estimate for dimensions $n$ such that this holds, namely for $n > c (\frac 1 {p_0-p} + \frac 1 {p-2})$. On the other hand, recently Eskenazis, Nayar and Tkocz \cite {ENT} proved that Ball's result is stable for $l_p^n$ and very large $p$: $(a^{(2)})^\perp \cap B_p^n$ is the maximal hyperplane section of $B_p^n$ for all dimensions, provided that $p_1 := 10^{15} \le p < \infty$. They call it "resilience of cubic sections". \\

Dual to hyperplane sections of convex bodies are projections of convex bodies onto hyperplanes. The known results for $l_p^n$-balls show a duality between sections and projections, when maximal and minimal directions $a$ and $p$ and the conjugate index $q = \frac p {p-1}$ are interchanged. Nevertheless the proofs in both situations are different, since volume does not behave well under duality. Barthe and Naor \cite{BN} determined the extremal hyperplane projections of $l_q^n$-balls except for the minimal hyperplane projections when $1 < q < 2$, corresponding to the dual maximal section case mentioned above when $2 < p < \infty$. For $q=1$, the projection of $B_1^n$ onto the hyperplane perpendicular to $a^{(1)} = (1,0 \etc 0)$ is maximal, the projection onto the hyperplane orthogonal to $a^{(2)}$ is minimal, which essentially is a consequence of Szarek's result \cite{S} on the best constants in the Khintchine inequality for $q=1$. Barthe and Naor \cite{BN} proved that this does no transfer to  $\frac 4 3 < q < 2$, at least, namely that the projection onto $a^{(n) \perp}$ has smaller volume than the one onto $a^{(2) \perp}$ for large dimensions $n$. In this case, we also give a quantitative estimate for dimensions $n$ when this happens, namely when $n > 5 (\frac 1 {p-\frac 4 3} + \frac 1 {2-p})$. Note that there is no complete duality here, since $\frac 4 3$ is not the dual index of $p_0 \simeq 26.265$.\\

For $1 \le p \le \infty$ and $n \in \N$, let $B_p^n$ denote the closed unit ball in $l_p^n$. Let $a \in S^{n-1} \subset \R^n$ be a direction vector. We introduce the normalized section function
$$A_{n,p}(a) := \frac{\vol_{n-1}(a^\perp \cap B_p^n)}{\vol_{n-1}(B_p^{n-1})} \ , $$
and the normalized projection function
$$P_{n,p}(a) := \frac{\vol_{n-1}(P_{a^\perp} (B_p^n))}{\vol_{n-1}(B_p^{n-1})} \ , $$
where $P_{a^\perp}$ denotes the orthogonal projection onto the hyperplane $a^\perp$. In terms of this notation, Ball's result states $A_{n,\infty}(a) \le A_{n,\infty}(a^{(2)})$ for all $a \in S^{n-1}$ and Eskenazis, Nayar and Tkocz' result reads $A_{n,p}(a) \le A_{n,p}(a^{(2)})$ for all $a \in S^{n-1}$ and $10^{15} \le p < \infty$. But as shown by Oleszkiewicz, $\lim_{n \to \infty} A_{n,p}(a^{(n)}) > A_{n,p}(a^{(2)})$ for $2 < p < p_0$. In the projection case, $P_{n,1}(a^{(2)}) \le P_{n,1}(a)$ for all $a \in S^{n-1}$, which Eskenazis, Nayar and Tkocz \cite{ENT} extended to $P_{n,q}(a^{(2)}) \le P_{n,q}(a)$ for all $a \in S^{n-1}$ and $1 < q \le 1 + 10^{-12}$. However, by Barthe and Naor \cite{BN}, $P_{n,q}(a^{(2)}) > \lim_{n \to \infty} P_{n,q}(a^{(n)})$ for $\frac 4 3 < q < 2$. \\
Our two main results study these limits in more detail.

\begin{theorem}\label{th1}
Let $2 < p < \infty$ and $n \in \N$. Then for all $2 < p < p_0 \simeq 26.265$
$$\lim_{n \to \infty} \frac{A_{n,p}(a^{(n)})}{A_{n,p}(a^{(2)})} = \sqrt{\frac 3 {\pi}} \sqrt{\frac {2^{\frac 2 p} \Gamma(1+\frac 1 p)^3}{\Gamma(1+\frac 3 p)}} > 1 \ . $$
We have the following quantitative estimate: $A_{n,p}(a^{(n)}) > A_{n,p}(a^{(2)})$ holds if \\

a) either $5 \le p < p_0$ and $n \ge \frac {650}{p_0-p}$  \;  or \; b)  $2 < p < 5$ and $n > \frac {65} {p-2}$ is satisfied.
\end{theorem}

{\bf Remarks}. (a) The constant $650$ in the statement for $5 \le p < p_0$ is not optimal, but by necessity fairly large since the $p$-derivative of
$f(p):=\sqrt{\frac 3 {\pi}} \sqrt{\frac {2^{\frac 2 p} \Gamma(1+\frac 1 p)^3}{\Gamma(1+\frac 3 p)}}$ at $p_0$ with $f(p_0)=1$ is small, $f'(p_0) \simeq - \frac 1 {1316}$. The derivative at $2$ is positive and larger in modulus, namely $f'(2) = \frac 1 4 (1 - \ln 2) \simeq \frac 1 {13}$. \\

(b) The case of complex hyperplane sections of $l_p^n(\C)$ is considered in \cite{JK}. \\

\begin{theorem}\label{th2}
Let $1 < q < 2$ and $n \in \N$. Then for all $\frac 4 3 < q < 2$
$$\lim_{n \to \infty} \frac{P_{n,q}(a^{(n)})}{P_{n,q}(a^{(2)})} = \sqrt{\frac 1 {\pi}} \sqrt{2^{\frac 2 q} \Gamma(\frac 1 q) \Gamma(2-\frac 1 q)} < 1 \ . $$
We have the following quantitative estimate: $P_{n,q}(a^{(n)}) < P_{n,q}(a^{(2)})$ holds if \\
$$ n > \frac{\frac{32} {15}}{q-\frac 4 3} + \frac{\frac{24} 5}{2-q} \ . $$
\end{theorem}

{\bf Remark}. For the derivative of $g(q) := \sqrt{\frac 1 {\pi}} \sqrt{2^{\frac 2 q} \Gamma(\frac 1 q) \Gamma(2-\frac 1 q)}$ we have $g'(\frac 4 3) = \frac 9 {32} (4 - \pi - 2 \ln 2) \simeq - \frac 1 {6.73}$ and $g'(2) = \frac 1 4 ( 1 - \ln 2) \simeq \frac 1 {13}$. \\

The limits in Theorems \ref{th1} and \ref{th2} were already determined by Oleszkiewicz \cite{O} and Barthe, Naor \cite{BN}.  Meyer and Pajor \cite{MP} showed that $A_{n,p}(a)$ is monotone increasing in $p$ for any fixed $n$ and $a$. Barthe and Naor proved that $P_{n,q}(a)$ is monotone increasing in $q$ for any fixed $n$ and $a$. \\

\section{Formulas}

Eskenazis, Nayar and Tkocz \cite{ENT}, Proposition 6, proved the following formula for the normalized volume of hyperplane sections.

\begin{proposition}
Let $1 \le p < \infty$, $n \in \N$ and $a = (a_j)_{j=1}^n \in S^{n-1} \subset \R^n$. Then
\begin{equation}\label{eq2.1}
A_{n,p}(a) = \Gamma(1+\frac 1 p) \ \E_{\xi,R} \frac 1 { ||\sum_{j=1}^n a_j R_j \xi_j||_2 } \ ,
\end{equation}
where $(\xi_j)_{j=1}^n$ are i.i.d. random vectors uniformly distributed on the sphere $S^2 \subset \R^3$ and $(R_j)_{j=1}^n$ are i.i.d. random variables with density $c_p^{-1} t^p \exp(-t^p)$ on $[0,\infty)$, $c_p := \frac 1 p \Gamma(1+\frac 1 p)$, independent of the  $(\xi_j)_{j=1}^n$.
\end{proposition}
For $p = \infty$ with $R_j=1$ one has $A_{n,\infty}(a) = \E_{\xi} \frac 1 { ||\sum_{j=1}^n a_j \xi_j||_2 }$, cf. K\"onig, Koldobsky \cite{KK}.
We will use another formula for $A_{n,p}(a)$ derived from \eqref{eq2.1}.

\begin{proposition}\label{prop1}
Let $1 \le p < \infty$, $n \in \N$ and $a = (a_j)_{j=1}^n \in S^{n-1} \subset \R^n$. Then
\begin{align}\label{eq2.2}
A_{n,p}(a) &= \Gamma(1+\frac 1 p) \frac 2 \pi \int_0^\infty \prod_{j=1}^n \gamma_p(a_j s) \ ds \ , \nonumber \\
\gamma_p(s) &:= \frac 1 {\Gamma(1+\frac 1 p)} \int_0^\infty \cos(sr) \exp(-r^p) \ dr \ .
\end{align}
\end{proposition}

\begin{proof}
Define $\textup{sinc}(x) := \frac{\sin x} x$, $\textup{sinc}(0) := 1$. Let $t>0$, $e \in S^2$ be fixed and $m$ denote the normalized Haar surface measure on $S^2$. Then
\begin{equation}\label{eq2.3}
\textup{sinc}(t) = \int_{S^2} \exp(i t <e,u>) dm(u)  \ .
\end{equation}
This implies for $(b_j)_{j=1}^n \subset \R^n$
\begin{align}\label{eq2.4}
\prod_{j=1}^n \textup{sinc}(b_j s) &= \int_{(S^2)^n} \exp( i s <e,\sum_{j=1}^n b_j u_j> ) \prod_{j=1}^n dm(u_j) \nonumber \\
&= \int_{(S^2)^n} \textup{sinc}( ||\sum_{j=1}^n b_j u_j||_2 s ) \prod_{j=1}^n dm(u_j) = \E_{\xi} \textup{sinc}( ||\sum_{j=1}^n b_j \xi_j||_2 s) \ ,
\end{align}
where the second equality follows from \eqref{eq2.3} by integration over $dm(e)$. Note that the first equality holds for all $e \in S^2$. \\
For all $t>0$ we have $\frac 2 \pi \int_0^\infty \textup{sinc}(ts) ds = \frac 1 t$ and \eqref{eq2.1} may be rewritten
\begin{align*}
A_{n,p}(a) & = \Gamma(1+\frac 1 p) \frac 2 \pi \E_{\xi,R} \int_0^\infty \textup{sinc}(||\sum_{j=1}^n a_j R_j \xi_j||_2 s ) ds \\
& = \Gamma(1+\frac 1 p) \frac 2 \pi \int_0^\infty \E_{\xi,R} \textup{sinc}(||\sum_{j=1}^n a_j R_j \xi_j||_2 s ) ds \ .
\end{align*}
The \textup{sinc}-integral is only a conditionally convergent Riemann integral. The verification that $\E_{\xi,R}$ and $\int_0^\infty$ may be interchanged is the same as in the proof of Proposition 3.2 (a) of K\"onig, Rudelson \cite{KR}. Using \eqref{eq2.4} and the independence of the $(R_j)_{j=1}^n$, we get
\begin{align*}
A_{n,p}(a) & = \Gamma(1 + \frac 1 p) \frac 2 \pi \int_0^\infty \E_R (\prod_{j=1}^n \textup{sinc}(a_j R_j s)) ds \\
& = \Gamma(1 + \frac 1 p) \frac 2 \pi \int_0^\infty \prod_{j=1}^n \E_{R_j} \textup{sinc}(a_j R_j s) ds \ .
\end{align*}
Denoting $\gamma_p(s) := \E_{R_1} \textup{sinc}(R_1 s)$, integration by parts gives
\begin{align*}
\gamma_p(s) & = c_p^{-1} \int_0^\infty \textup{sinc}(sr) r^p \exp(-r^p) dr \\
& = c_p^{-1} \frac 1 p \int_0^\infty \cos(sr) \exp(-r^p) dr = \frac 1 {\Gamma(1+\frac 1 p)} \int_0^\infty \cos(sr) \exp(-r^p) dr \ .
\end{align*}
\end{proof}

Equation \eqref{eq2.1} yields $A_{n,p}(a^{(2)}) = 2^{\frac 1 2-\frac 1 p}$, cf. \cite{ENT}, section 3.2. \\

{\bf Remarks.} (a) Proposition \ref{prop1} is also found in Koldobsky \cite{K}, Theorem 3.2, with a different proof. \\
(b) For $1 \le p \le 2$ the $\gamma_p$ are just the (positive) $p$-stable random variables. In the case interesting for us, namely $2 < p < \infty$, the variables $\gamma_p$ take positive and negative values. For $p \notin 2 \N$, $\gamma_p$ has only finitely many real zeros, see P\'olya \cite{Po}, whereas for $p \in 2 \N$, $\gamma_p$ has infinitely many real zeros, see Boyd \cite{Bo} . \\

Barthe and Naor \cite{BN} proved the following formula for the volume of the orthogonal projection of $B_q^n$ onto hyperplanes.

\begin{proposition}
Let $1 \le q < \infty$, $p:= \frac q {q-1}$ be the conjugate index, $n \in \N$ and $a = (a_j)_{j=1}^n \in S^{n-1} \subset \R^n$. Then
\begin{equation}\label{eq2.5}
P_{n,q}(a) = \Gamma(\frac 1 q) \ \E |\sum_{j=1}^n a_j X_j| \ ,
\end{equation}
where the $X_j$ are i.i.d. symmetric random variables with density function \\ $d_q^{-1} |t|^{p-2} \exp(-|t|^p)$, $t \in \R$, $d_q = \frac 2 p \Gamma(\frac 1 q)$. A second formula for $P_{n,q}(a)$ is
\begin{align}\label{eq2.6}
P_{n,q}(a) & = \Gamma(\frac 1 q) \frac 2 \pi \int_0^\infty \frac{1 - \prod_{j=1}^n \delta_q(a_j s)}{s^2} ds \ , \nonumber \\
\delta_q(s) & := \frac p {\Gamma(\frac 1 q)} \int_0^\infty \cos(s r) \ r^{p-2} \exp(-r^p) dr \ .
\end{align}
\end{proposition}

Note that $\E |X_1| = \frac 1 {\Gamma(\frac 1 q)}$. To deduce \eqref{eq2.6} from\eqref{eq2.5}, apply the usual formula $|x|= \frac 1 \pi \int_\R \frac {1 - Re(\exp(ixs))}{s^2} ds$ \ to find
\begin{align*}
\E |\sum_{j=1}^n a_j X_j| & = \frac 1 \pi \int_\R \frac{1 - \E \exp(i (\sum_{j=1}^n a_j X_j) s)}{s^2} ds =
\frac 1 \pi \int_\R \frac{1 -  \prod_{j=1}^n \E \exp(i  a_j X_j s))}{s^2} ds \\
& = \frac 2 \pi \int_0^\infty \frac{1 -  \prod_{j=1}^n \E \cos(a_j X_j s))}{s^2} ds = \frac 2 \pi \int_0^\infty \frac{1 -  \prod_{j=1}^n \delta_q(a_j s)}{s^2} ds \ , \\
\delta_q(s) &= \frac p {\Gamma(\frac 1 q)} \int_0^\infty \cos(sr) \ r^{p-2} \exp(-r^p) dr \ .
\end{align*}

Differentiation and integration by parts yields a relation between the functions $\delta_q$ and $\gamma_p$ in \eqref{eq2.6} and \eqref{eq2.2}:
\begin{align}\label{eq2.7}
\delta_q'(s) &= - \frac p {\Gamma(\frac 1 q)} \int_0^\infty \sin(sr) \ r^{p-1} \exp(-r^p) dr  \nonumber \\
& = - \frac s {\Gamma(\frac 1 q)} \int_0^\infty \cos(sr) \exp(-r^p) dr = - \frac{\Gamma(1+ \frac 1 p)}{\Gamma(1-\frac 1 p)} \ s \ \gamma_p(s) \ .
\end{align}
Since $\gamma_4''(s) = - \frac{\Gamma(\frac 3 4)} {4 \Gamma(\frac 5 4)} \ \delta_{\frac 4 3}(s)$, we have $\gamma_4'''(s) = \frac 1 4 s \gamma_4(s)$. Similarly, for all $k \in \N$, $\gamma_{2k}^{(2k-1)}(s)= (-1)^k \frac 1 {2 k} s \gamma_{2 k}(s)$. Therefore the functions $\gamma_{2k}$ studied by Boyd \cite{Bo} satisfy a linear differential equation. \\

For $q \searrow 1$, the variables $X_j$ tend to the Rademacher variables with $\delta_q(s) \to \delta_1(s) = \cos(s)$, and the best constants in the Khintchine inequality, which were determined by Szarek \cite{S}, yield the extrema of $P_{n,1}$: $a^{(2)}$ for the minimum and $a^{(1)}$ for the maximum. \\

\vspace{0,8cm}

\section{Prerequisites for the proof of Theorem \ref{th1}}

For the proof of Theorems \ref{th1} we need two lemmas on $\Gamma$-functions.

\begin{lemma}\label{lem1}
(a) Let $f(p) := \frac{\Gamma(1+\frac 3 p)}{\Gamma(1+\frac 1 p)}$. Then $f(p) \ge 0.9429$ for all $3 \le p < \infty$. \\
(b) Let $g(p) := \Big( \frac 3 \pi \frac { 2^{\frac 2 p} \Gamma(1+\frac 1 p)^3 }{ \Gamma(1+\frac 3 p) } \Big)^{\frac 1 2}$. Then there is exactly one solution
$p_0 \in (2,\infty)$ of $g(p) = 1$, $p_0 \simeq 26.265$. For all $2 < p < p_0$ we have $g(p) > 1$. The function $g'$ has exactly one zero $p_1 \in [2,\infty)$, $p_1 \simeq 4.192$. For $2 \le p < p_1$, $g$ is strictly increasing, for $p_1 < p < \infty$, $g$ is strictly decreasing. The following lower estimates hold:
$$g(p) \ge 1 + \frac{p_0-p}{1317} \ , \ p \in [5,p_0] \; , \; g(p) > \frac{25}{24} \ , \ p \in[4,5] \; , \;  g(p) \ge 1 + \frac{p-2} {44} \ , \ p \in [2,4] \ . $$
\end{lemma}

\begin{proof}
(a) In terms of the Digamma function $\Psi := (\ln \Gamma)'$ we have
$$f'(p) = \frac{f(p)}{p^2} (\Psi(1+\frac 1 p) - 3 \Psi(1+\frac 3 p)) \ . $$
For $F(p) := \Psi(1+\frac 1 p) - 3 \Psi(1+\frac 3 p)$ one has $F'(p) = \frac 1 {p^2} (9 \Psi'(1+\frac 3 p) - \Psi'(1+\frac 1 p))$.
By Abramowitz, Stegun \cite{AS}, 6.3.16 and 6.4.10 for all $x >0$
\begin{equation}\label{eq3.1}
\Psi(1+x) = -\gamma + \sum_{n=1}^\infty \frac x {n(n+x)} \; , \; \Psi'(1+x) = \sum_{n=1}^\infty \frac 1 {(n+x)^2} \ ,
\end{equation}
where $\gamma \simeq 0.5772$ denotes the Euler constant. Therefore $\Psi'$ is decreasing, and we conclude for all $0 \le x \le 1$ that
$\frac{\pi^2} 6 -1 = \Psi'(2) \le \Psi'(1+x) \le \Psi'(1) = \frac{\pi^2} 6$.
Hence $F'(p) \ge \frac 1 {p^2} (\frac{4 \pi^2} 3 -9) > 0$ for all $p \ge 3$. Thus $F$ is increasing. Since $F(9) \simeq -0.012$, $F(10) \simeq 0.084$, $F$ has exactly one zero $p_1 \in(3,\infty)$, $p_1 \simeq 9.115$. Hence $f$ is decreasing in $(3,p_1)$ and increasing in $(p_1,\infty)$. For all $p \ge 3$, $f(p) \ge f(p_1) > 0.9429$. \\

(b) Let $h(p) := \frac { 2^{\frac 2 p} \Gamma(1+\frac 1 p)^3 }{ \Gamma(1+\frac 3 p) }$. Then
$$h'(p) = \frac{h(p)}{p^2} ( 3 \Psi(1+\frac 3 p) - 3 \Psi(1+\frac 1 p) - 2 \ln 2 ) \ . $$
By \eqref{eq3.1} and the geometric series we find for $p > 3$ \\
\begin{align*}
(\ln h)'(p) &= \frac{h'(p)}{h(p)} = \frac 1 {p^2} ( 3 \sum_{n=1}^\infty ( \frac{\frac 3 p}{n(n+\frac 3 p)} - \frac{\frac 1 p}{n(n+\frac 1 p)} ) - 2 \ln 2 ) \\
&= \frac 1 {p^2} ( 3 \sum_{n=1}^\infty \sum_{k=0}^\infty \frac{(-1)^k}{n^{k+2}} \frac{3^{k+1}-1}{p^{k+1}} - 2 \ln 2 ) \\
&= \frac 1 {p^2} ( 3 \sum_{k=0}^\infty (-1)^k \zeta(k+2) \frac{3^{k+1}-1}{p^{k+1}} - 2 \ln 2 ) \ .
\end{align*}
The sum is an alternating series with decreasing coefficients. We find that
\begin{equation}\label{eq3.2}
(\ln h)'(p) \le -\frac{2 \ln 2}{p^2} + \frac{\pi^2}{p^3} - \frac{24 \zeta(3)}{p^4} + \frac{13}{15} \frac{\pi^4}{p^5} - \frac{240 \zeta(5)} {p^6} + \frac{242}{315} \frac{\pi^6}{p^7} < 0
\end{equation}
holds for all $5 \le p < \infty$. Therefore $\ln h$, $h$ and $g(p) = \sqrt{\frac 3 \pi h(p)}$ are strictly decreasing in $[5,\infty)$. We have
$\lim_{p \to \infty} g(p) = \sqrt{\frac 3 \pi} < 1$, $g(26) \simeq 1.00020$, $g(27) \simeq 0.99945$: There is exactly one $p_0 \in [5,\infty)$ with $g(p_0)=1$, $p_0 \simeq 26.265$, and for $5 \le p < p_0$ we have $g(p) > 1$. Inequality \eqref{eq3.2} yields for $5 \le p \le p_0$ that
$(\ln h)'(p) \le - \frac{1.04768}{p^2}$. Hence for these $p$
$$g'(p) = g(p) (\ln g)'(p) = \frac 1 2 g(p) (\ln h)'(p) \le - \frac 1 2 \frac{1.04768}{p^2} \le - \frac 1 2 \frac{1.04768}{p_0^2} < - \frac 1 {1317} \ . $$
This implies $g(p) \ge 1 + \frac 1{1317} (p_0-p)$ for all $5 \le p \le p_0$. \\

To show $g(p) > 1$ also for $2 < p < 5$, note that $g(2)=1$ and
$$g'(p) = \frac 1 2 g(p) (\ln h)'(p) = \frac 3 2 \frac{g(p)}{p^2} ( \Psi(1+\frac 3 p) - \Psi(1+\frac 1 p) - \frac 2 3 \ln 2 ) \ . $$
Again by \eqref{eq3.1}
$$\Psi(1+\frac 3 p) - \Psi(1+\frac 1 p) = \sum_{n=1}^\infty (\frac{\frac 3 p}{n (n + \frac 3 p)} - \frac{\frac 1 p}{n (n + \frac 1 p)}) =
\sum_{n=1}^\infty \frac{2 p}{(np+1)(np+3)} \ . $$
All summands are decreasing in $p$. Thus $k(p) := \Psi(1+\frac 3 p) - \Psi(1+\frac 1 p) - \frac 2 3 \ln 2$ is strictly decreasing in $p$, with
$k(4) = \pi - \frac 8 3 - \frac 2 3 \ln 2 \simeq 0.0128 >0$ and $k(5) \simeq -0.0470 <0$. Thus $g'$ has exactly one zero $p_1 \in (2,\infty)$, $p_1 \simeq 4.193$, and $g$ is strictly increasing in $(2,p_1)$ and strictly decreasing in $(p_1,\infty)$. We know already that $g(5) > 1$ and hence $g(p)>1$ for all $2 < p \le 5$. For $p \in [4,5]$, $g(p) \ge \min(g(4),g(5)) = g(5) > \frac {25}{24}$. Further
$$(\frac{g(p)}{p^2})' = \frac 3 2 \frac{g(p)}{p^4} (k(p) - \frac 4 3 p) < 0 \ , $$
since $k(p)- \frac 4 3 p \le k(2)-\frac 8 3 = - (2+ \frac 2 3 \ln(2)) < 0$. Therefore $\frac{g(p)}{p^2}$ and $k(p)$ are both strictly decreasing and positive for $p \in [2,p_1)$, and with $g'(p) = \frac{g(p)}{p^2} k(p)$, $g'$ is decreasing and hence $g$ is concave in $[2,p_1]$. Therefore for $2 \le p \le 4$
$$g(p) \ge 1 + \frac{g(4)-1} 2 (p-2) > 1 + \frac{p-2}{44} \ , $$
which proves all lower estimates stated in Lemma \ref{lem1}.
\end{proof}

For $p \to \infty$, the functions $\gamma_p$ in \eqref{eq2.2} tend to $\gamma_\infty$, $\gamma_\infty(s) = \textup{sinc}(s)$. We estimate their difference for $p \ge 2$.

\begin{lemma}\label{lem2}
Let $2 \le p < \infty$. Then for all $s > 0$
$$ | \textup{sinc}(s) - \int_0^\infty \cos(s r) \exp(-r^p) dr | \le 0.3926 \ . $$
This implies $\gamma_p(s) > 0$ for all $0 \le s \le \frac 2 3 \pi$.
\end{lemma}

\begin{proof}
We have $\int_0^\infty \exp(-s^p) ds = \Gamma(1 + \frac 1 p) < 1$. Since $\textup{sinc}(s) = \int_0^1 \cos(sr) dr$, we find
\begin{align*}
| \textup{sinc} (s) & - \int_0^\infty \cos(s r) \exp(-r^p) dr | \\
& = | \int_0^1 \cos(s r) (1-\exp(-r^p)) dr - \int_1^\infty \cos(s r) \exp(-r^p) dr | \\
& \le (1 - \Gamma(1 + \frac 1 p)) + 2 \int_1^\infty \exp(-r^p) dr \\
& = (1 - \Gamma(1 + \frac 1 p)) + \frac 2 p \int_1^\infty u^{\frac 1 p -1} \exp(-u) du =: \phi(p) \ .
\end{align*}
Then $\phi' < 0$, since for $p \ge 2$
\begin{align*}
\phi'(p) & = - \frac 1 {p^2} ( 2 \int_1^\infty u^{\frac 1 p -1} (1+\frac{\ln(u)} p) \exp(-u) du - \Gamma(1 + \frac 1 p) \Psi(1+\frac 1 p) ) \\
&\le - \frac 1 {p^2} ( 2 \int_1^\infty u^{-\frac 1 2} \exp(-u) du - \frac 1 {100} ) < - \frac 2 5 \frac 1 {p^2} <0 \ ,
\end{align*}
using that $\int_1^\infty u^{-\frac 1 2} \exp(-u) du \simeq 0.219$ and $\Psi(1+\frac 1 p) < 0$ for all $p > \frac{13} 6$ and
$\Gamma(1+\frac 1 p) |\Psi(1+\frac 1 p)| < \frac 1 {100}$ for $2 \le p \le \frac{13} 6$. Therefore $\phi(p) \le \phi(2) < 0.3926$. \\

This yields for all $0 \le s \le \frac 2 3 \pi$ and $p \ge 2$
$$\Gamma(1+\frac 1 p) \gamma_p(s) = \int_0^\infty \cos(sr) \exp(-r^p) dr \ge \textup{sinc}(s) - 0.3926 \ge \frac{3 \sqrt 3}{4 \pi} -0.3926 > \frac 1 {50} > 0 \ . $$
\end{proof}

{\bf Remark}. The derivative of $p \phi(p)$ is increasing with
$$\lim_{p \to \infty} (p \phi(p))' = \gamma + 2 \int_1^\infty \frac 1 u \exp(-u) du \le 1.016 \ . $$
Thus for all $p \ge 1$, $|\textup{sinc}(s) - \int_0^\infty \cos(s r) \exp(-r^p) dr | \le \frac{1.016} p$. Actually, $\gamma_p(s) > 0$ for all $p \ge 1$ and $s \in [0, \pi]$. However, we do not need this. \\

\section{Proof of Theorem \ref{th1}}

{\bf Proof of Theorem \ref{th1}. } \\
(i) To estimate $A_{n,p}(a^{(n)})$ from below, we first find a lower bound for $\gamma_p(\frac s {\sqrt n})$ for all $s \le \frac 3 2 \sqrt n$. By the series representation for $\cos x$ we have for $0 \le x \le \frac 3 2$
$$\cos x - (1 - \frac {x^2} 2 + \frac {x^4} {26} ) = \frac{x^4}{312} - \frac{x^6}{720} + \sum_{k=4}^\infty (-1)^m \frac{x^{2m}}{(2m)!} > 0 \ , $$
$\cos x >0$, $1 - \frac {x^2} 2 + \frac {x^4} {26} > 0$. Therefore for $s \le \frac 3 2 \sqrt n$
$$\gamma_p(\frac s {\sqrt n}) \ge \frac 1 {\Gamma(1+\frac 1 p)} \Big[ \int_0^{\frac 3 2 \frac{\sqrt n} s} (1 - \frac {s^2 r^2} {2n} + \frac {s^4 r^4} {26 n^2}) \exp(-r^p) dr + \int_{\frac 3 2 \frac{\sqrt n} s}^\infty \cos(\frac{sr}{\sqrt n}) \exp(-r^p) dr \Big] \ . $$
To estimate this from below, write $\int_0^{\frac 3 2 \frac{\sqrt n} s} = \int_0^\infty - \int_{\frac 3 2 \frac{\sqrt n} s}^\infty$ and use that
$|\cos(x)| \le 1$,
\begin{align}\label{eq3.3}
\gamma_p(\frac s {\sqrt n}) &\ge \frac 1 {\Gamma(1+\frac 1 p)} [ \int_0^\infty (1 - \frac{s^2 r^2}{2n} + \frac{s^4 r^4}{26 n^2}) \exp(-r^p) dr - R ] \nonumber \\
& = \frac 1 {\Gamma(1+\frac 1 p)} [\Gamma(1+\frac 1 p) - \frac{\Gamma(1+\frac 3 p) s^2}{6 n} + \frac{\Gamma(1+\frac 5 p) s^4}{130 n^2} - R ] \ ,
\end{align}
where
\begin{align}\label{eq3.4}
R &:= \int_{\frac 3 2 \frac{\sqrt n} s}^\infty (2 - \frac {s^2 r^2} {2n} + \frac {s^4 r^4} {26 n^2}) \exp(-r^p) dr \nonumber \\
& = \frac 1 p \int_{(\frac 3 2 \frac{\sqrt n} s)^p}^\infty (2 u^{\frac 1 p - 1} - \frac{s^2} {2n} u^{\frac 3 p - 1} + \frac{s^4}{26 n^2} u^{\frac 5 p - 1}) \exp(-u) du \; ; \; u \ge 1 \ .
\end{align}

(ii) We first consider the case $p \ge 5$. Since $u \ge 1$ in the above integral, $u^{\frac 5 p -1} \le 1$, $2 u^{\frac 1 p -1} \le 2 (\frac 2 3 \frac s {\sqrt n})^{p-1}$ and $\int_{(\frac 3 2 \frac{\sqrt n} s)^p} \exp(-u) du = \exp(-\frac 3 2 \frac{\sqrt n} s)^p)$.
The remainder term $R$ will be smaller than the fourth order term $\frac{\Gamma(1+\frac 5 p) s^4}{130 n^2}$ provided that
$$\frac 1 {p \Gamma(1+\frac 5 p)} (2 (\frac 2 3 \frac s {\sqrt n})^{p-1} + \frac{s^4}{26 n^2}) \exp(-(\frac 3 2 \frac{\sqrt n} s)^p) < \frac{s^4}{130 n^2} \ . $$
For $s \le \frac 3 2 \sqrt n$ the left side is decreasing in $p$, and therefore this condition is the strongest for $p=5$. Writing $y:=\frac s {\sqrt n}$, it means
$$\frac 1 5 ( \frac{32}{81} y^4 + \frac 1 {26} y^4) \exp(-(\frac 3 2 \frac 1 y)^5) < \frac 1 {130} y^4 $$
or $\frac{913}{81} < \exp((\frac 3 2 \frac 1 y)^5)$, $y < \frac 3 2 \frac 1 {\ln(\frac{913}{81})^(1/5)} \simeq 1.2567$.
Choosing $s \le \frac 7 6 \sqrt n$, $R \le \frac{\Gamma(1+\frac 5 p) s^4}{130 n^2}$ is satisfied and therefore
$$\gamma_p(\frac s {\sqrt n}) \ge 1 - c \frac {s^2} n \quad , \quad  c:= \frac 1 6 \frac{\Gamma(1+\frac 3 p)}{\Gamma(1+\frac 1 p)} \ . $$
By the proof of Lemma \ref{lem1} (a), $\frac{\Gamma(1+\frac 3 p)}{\Gamma(1+\frac 1 p)}$ is decreasing for $5 \le p \le p_1 \simeq 9.115$ and increasing for $p > p_1$. Its value at $p_0$ is less than the one at $5$, so that $c \le 0.1622 < \frac 8 {49}$, for its value at $p=5$. Then for $s \le \frac 7 6 \sqrt n$,
$x := c \frac{s^2} n \le \frac 8 {49} \frac{49}{36} = \frac 2 9$. We have
$$\ln(1-x) = - \sum_{j=1}^\infty \frac{x^j} j \ge -x - \frac 1 2 x^2 \sum_{k=0}^\infty x^k = -x - \frac 1 2 \frac{x^2}{1-x} \ge -x-\frac 9 {14} x^2 \ , $$
and hence
$$\gamma_p(\frac s {\sqrt n})^n \ge (1 - c \frac{s^2} n)^n \ge \exp(-c s^2 - \frac 9 {14} c^2 \frac{s^4} n) \ge \exp(-c s^2) (1 - \frac 9 {14} c^2 \frac{s^4} n) \ . $$
By Lemma \ref{lem1} (a),  $\frac{\Gamma(1+\frac 3 p)}{\Gamma(1+\frac 1 p)} \ge 0.9429$ and hence $c \ge 0.1571$ and for $s \le \frac 7 6 \sqrt n$
\begin{align*}
\int_0^{\frac 7 6 \sqrt n} & \gamma_p(\frac s {\sqrt n})^n ds \ge \int_0^{\frac 7 6 \sqrt n} \exp(-c s^2) (1 - \frac 9 {14} c^2 \frac{s^4} n ) ds  \\
& = \int_0^\infty \exp(-c s^2) (1 - \frac 9 {14} c^2 \frac{s^4} n ) ds - \int_{\frac 7 6 \sqrt n}^\infty \exp(-c s^2) (1 - \frac 9 {14} c^2 \frac{s^4} n ) ds \end{align*}
For $s \ge \frac 7 6 \sqrt n$ and $c \ge 0.1571$ we have $1 - \frac 9 {14} c^2 \frac{s^4} n <0$ for all $n \ge 35$. Actually, evaluating the last integral in terms of the error function shows that the integral is already negative for $n \ge 24$. Thus for $n \ge 24$
\begin{align*}
\int_0^{\frac 7 6 \sqrt n} \gamma_p(\frac s {\sqrt n})^n ds& \ge \int_0^\infty \exp(-c s^2) (1 - \frac 9 {14} c^2 \frac{s^4} n ) ds  \\
& = \frac 1 2 \sqrt{\frac \pi c} (1 - \frac {27}{56} \frac 1 n) = \sqrt{\frac{3 \pi} 2} \sqrt{\frac{\Gamma(1+\frac 1 p)}{\Gamma(1+\frac 3 p)}} (1 - \frac {27}{56} \frac 1 n) \ ,
\end{align*}
where we used that $\int_0^\infty \exp(-c s^2) ds = \frac 1 2 \sqrt{\frac \pi c}$ and $\int_0^\infty \exp(-c s^2) c^2 s^4 ds = \frac 3 8 \sqrt{\frac \pi c}$.
By Lemma \ref{lem2} we have $\gamma_p(s) > 0$ for all $0 \le s \le 2$. Hence
$$0 < \sqrt n \int_{\frac 7 6}^2 \gamma_p(s)^n du = \int_{\frac 7 6 \sqrt n}^{2 \sqrt n} \gamma_p(\frac s {\sqrt n})^n ds  \ . $$
By the proof of Proposition \ref{prop1}
\begin{align*}
|\gamma_p(s)| & = |\frac p {\Gamma(1+\frac 1 p)} \int_0^\infty \textup{sinc}(sr) r^p \exp(-r^p) dr| \\
& = |\frac 1 {s \Gamma(1+\frac 1 p)} \int_0^\infty \sin(sr) \ p r^{p-1} \exp(-r^p) dr| \ , \ |\sin(sr)| \le 1  \\
& \le \frac 1 {s \Gamma(1+\frac 1 p)} \int_0^\infty \exp(-u) du = \frac 1 {s \Gamma(1+\frac 1 p)} \ .
\end{align*}
This yields the tail estimate for $p \ge 5$
\begin{align*}
\int_{2 \sqrt n}^\infty |\gamma_p(\frac s {\sqrt n})|^n ds &= \sqrt n \int_2^\infty |\gamma_p(s)|^n ds  \\
& \le \frac{\sqrt n}{\Gamma(1+\frac 1 p)^n} \int_2^\infty s^{-n} ds = \frac{2 \sqrt n}{n-1} (\frac 1 {2 \Gamma(1+\frac 1 p)})^n < \frac{2 \sqrt n}{n-1} 0.5446^n \ .
\end{align*}
We conclude for $p \ge 5$ and $n \ge 24$ that
\begin{align*}
\int_0^\infty \gamma_p(\frac s {\sqrt n})^n ds & \ge \int_0^{\frac 7 6 \sqrt n} \gamma_p(\frac s {\sqrt n})^n ds - \int_{2 \sqrt n}^\infty |\gamma_p(\frac s {\sqrt n})|^n ds \\
& \ge \sqrt{\frac{3 \pi} 2} \sqrt{\frac{\Gamma(1+\frac 1 p)}{\Gamma(1+\frac 3 p)}} (1 - \frac {27}{56} \frac 1 n) - \frac{2 \sqrt n}{n-1} 0.5446^n \\
& \ge \sqrt{\frac{3 \pi} 2} \sqrt{\frac{\Gamma(1+\frac 1 p)}{\Gamma(1+\frac 3 p)}} (1 - \frac {27}{56} \frac 1 n - \frac{\sqrt n}{n-1} 0.5446^n) \ ,
\end{align*}
using $\sqrt{\frac{3 \pi} 2} \sqrt{\frac{\Gamma(1+\frac 1 p)}{\Gamma(1+\frac 3 p)}} > 2$. For $n \ge 24$, $\frac{\sqrt n}{n-1} 0.5446^n < \frac{10^{-5}} n$ and $\frac {27}{56} + 10^{-5} < \frac{193}{400}$, so that
\begin{align*}
A_{n,p}(a^{(n)}) & = \Gamma(1+\frac 1 p) \frac 2 \pi \int_0^\infty \gamma_p(\frac s {\sqrt n})^n ds \\
&\ge \sqrt{\frac 6 \pi} \sqrt{\frac{\Gamma(1+\frac 1 p)^3}{\Gamma(1+\frac 1 p)}} (1 - \frac{193}{400} \frac 1 n) \ .
\end{align*}
This is  $> A_{n,p}(a^{(2)}) = 2^{\frac 1 2 - \frac 1 p}$, provided that
$$\frac{A_{n,p}(a^{(n)})} {A_{n,p}(a^{(2)})} \ge \sqrt{\frac 3 \pi} \sqrt{\frac{2^{\frac 2 p} \Gamma(1+\frac 1 p)^3}{\Gamma(1+\frac3 p)}} (1 - \frac{193}{400} \frac 1 n) > 1 \ . $$
By Lemma \ref{lem1} (b), the quotient $g(p) := \sqrt{\frac 3 \pi} \sqrt{\frac{2^{\frac 2 p} \Gamma(1+\frac 1 p)^3}{\Gamma(1+\frac3 p)}}$ is $> 1$ for all $2 < p < p_0 \simeq 26.265$, with $g(p) \ge 1 + \frac 1 {1317} (p_0-p)$ for all $5 \le p \le p_0$. We find for $p \ge 5$ and $n \ge 24$,
$$\frac{A_{n,p}(a^{(n)})} {A_{n,p}(a^{(2)})} \ge (1 + \frac 1 {1317} (p_0-p)) (1-\frac{193}{400} \frac 1 n)  \ . $$
This is $>1$ provided that $5 \le p \le p_0$ and $n \ge \frac{650}{p_0-p}$; $n \ge 24$ being automatically satisfied. \\

(iii) Secondly we consider the case $2 < p \le 5$. To estimate the remainder $R$ in \eqref{eq3.4}, we use that in this case $u^{\frac 5 p -1} \le u^{\frac 3 2}$. For $x > 0$
\begin{align*}
\int_x^\infty u^{\frac 3 2} \exp(-u) du & \le \Big( (\int_x^\infty u \exp(-u) du ) (\int_x^\infty u^2 \exp(-u) du ) \Big)^{\frac 1 2} \\
& = \Big( (1+x) (2+2 x +x^2) \Big)^{\frac 1 2} \exp(-x) \ ,
\end{align*}
which is $ \le \frac {13}{20} x^2 \exp(-x)$ for all $x \ge \frac 9 2$. Now choose $s \le \sqrt{\frac n 2}$ for $2 \le p \le 5$. Then $x := (\frac 3 2 \frac{\sqrt n} s)^p \ge (\frac 3 {\sqrt 2})^p \ge \frac 9 2$ and
$$ R \le \frac 1 p ( 2 (\frac 2 3 \frac s {\sqrt n})^{p-1} + \frac{s^4}{26 n^2} \frac {13}{20} (\frac 3 2 \frac {\sqrt n} s)^{2p} ) \exp(-(\frac 3 2 \frac{\sqrt n} s)^p ) \ . $$
Again we want this to be smaller than the fourth order term $\Gamma(1+\frac 5 p) \frac{s^4}{130 n^2}$, a condition which is strongest for $p=2$. We then require for $y := \frac s {\sqrt n}$
$$\frac 1 2(\frac 4 3 y + \frac{81}{640}) \exp(-(\frac 3 2 \frac 1 y)^2) < \Gamma(\frac 7 2) \frac {y^4}{130 n^2} \ , $$
which is satisfied for all $y \le 0.7161$, and in particular for our choice $y = \frac s {\sqrt n} \le \frac 1 {\sqrt 2}$. Therefore for $s \le \sqrt{\frac n 2}$, as in part (ii),
$$\gamma_p(\frac s {\sqrt n}) \ge 1 - c \frac{s^2} n \quad , \quad c:= \frac 1 6 \frac{\Gamma(1+\frac 3 p)}{\Gamma(1+\frac 1 p)} \ . $$
We have $c \le \frac 1 4$ for $2 \le p \le 5$ and $x:= c \frac{s^2} n \le \frac 1 8$. Similarly as in (ii), \\
$\ln(1-x) \ge -x-\frac{x^2}{1-x} \ge -x-\frac 4 7 x^2$ and
\begin{align*}
\gamma_p(\frac s {\sqrt n}) & \ge (1-c \frac{s^2} n)^n =\exp(n \ln (1-c \frac{s^2} n)) \\
& \ge \exp(-c s^2 - \frac 4 7 c^2 \frac{s^4} n) \ge \exp(-c s^2) (1- \frac 4 7 c^2 \frac{s^4} n)
\end{align*}
and
\begin{align*}
\int_0^{\sqrt{\frac n 2}} & \gamma_p(\frac s {\sqrt n})^n ds \ge \int_0^{\sqrt{\frac n 2}} \exp(-c s^2) (1 - \frac 4 7 c^2 \frac{s^4} n ) ds  \\
& \ge \int_0^\infty \exp(-c s^2) (1 - \frac 4 7 c^2 \frac{s^4} n ) ds - \int_{\sqrt{\frac n 2}}^\infty \exp(-c s^2) ds \\
& = \frac 1 2 \sqrt{\frac \pi c} (1 - \frac 3 7 \frac 1 n) - \int_{\sqrt{\frac n 2}}^\infty \exp(-c s^2) ds \ .
\end{align*}
To estimate the error term, note that $c \ge .16219$ -its value for $p=5$- and
\begin{align*}
\int_{\sqrt{\frac n 2}}^\infty \exp(-c s^2) ds & = \frac 1 {2 \sqrt c} \int_{\frac c 2 n}^\infty \frac 1 {\sqrt u} \exp(-u) du \\
& \le \frac 1 {c \sqrt{2 n}} \int_{\frac c 2 n}^\infty \exp(-u) du = \frac 1 {c \sqrt{2 n}} \exp(-\frac c 2 n) \le \frac{4.36}{\sqrt n} 0.9222^n \ .
\end{align*}
Again $\int_{\sqrt{\frac n 2}}^{2 \sqrt n} \gamma_p(\frac s {\sqrt n})^n ds \ge 0$, since by Lemma \ref{lem2} $\gamma_p(x) >0$ for all $0 \le x \le 2$ and, as in (ii), $\int_{2 \sqrt n}^\infty |\gamma_p(\frac s {\sqrt n})|^n ds \le \frac {2 \sqrt n}{n-1} 0.5446^n$, so that for $2 \le p \le 5$
$$\int_0^\infty \gamma_p(\frac s {\sqrt n})^n ds \ge  \frac 1 2 \sqrt{\frac \pi c} (1 - \frac 3 7 \frac 1 n) -(\frac{4.36}{\sqrt n} 0.9222^n + \frac {2 \sqrt n}{n-1} 0.5446^n) \ . $$
Since $c \le \frac 1 4$, $\frac 1 2 \sqrt{\frac \pi c} \ge \sqrt{\pi}$. Further for $n \ge 33$
$$\frac 3 7 + \frac 1 {\sqrt \pi} (4.36 \sqrt n \ 0.9222^n + \frac {2 n \sqrt n}{n-1} 0.5446^n) \le 1.405 \ , $$
so that $\int_0^\infty \gamma_p(\frac s {\sqrt n})^n ds \ge \frac 1 2 \sqrt{\frac \pi c} (1- \frac{1.405} n)$. For $p \in [2,4]$, using Lemma \ref{lem1} (b),
\begin{align*}
\frac{A_{n,p}(a^{(n)})} {A_{n,p}(a^{(2)})} &\ge \sqrt{\frac 3 \pi} \sqrt{\frac{2^{\frac 2 p} \Gamma(1+\frac 1 p)^3}{\Gamma(1+\frac3 p)}} (1 - \frac {1.405} n) \\
& \ge (1 + \frac{p-2}{44}) (1 - \frac{1.405} n) \ .
\end{align*}
For $n \ge \frac{65}{p-2}$ this is $>1$, with $n \ge 33$ being automatically satisfied. For $p \in[4,5]$, again by Lemma \ref{lem1} (b),
$\sqrt{\frac 3 \pi} \sqrt{\frac{2^{\frac 2 p} \Gamma(1+\frac 1 p)^3}{\Gamma(1+\frac3 p)}} > \frac{25}{24}$ and $\frac{A_{n,p}(a^{(n)})} {A_{n,p}(a^{(2)})} > 1$ is satisfied for $n \ge \frac{65}{p-2}$, too.
\hfill $\Box$  \\

\section{Prerequisites for the proof of Theorem \ref{th2}}

We need two lemmas for the proof of Theorem \ref{th2}.

\begin{lemma}\label{lem3}
(a) For $q \in [1,2]$, let $f(q) := \frac{\Gamma(2-\frac 1 q)}{\Gamma(\frac 1 q)}$. Then $f$ is decreasing, with $f(1)=1$, $f(2)=\frac 1 2$ and $f(\frac 4 3) \le 0.7397$. \\
(b) For $q \in [1,2]$, let $g(q) := \sqrt{\frac{2^{\frac 2 q}}\pi \Gamma(\frac 1 q) \Gamma(2-\frac 1 q)}$. Then $g'$ has exactly one zero in $q_1 \in (1,2)$, $q_1 \simeq 1.612$, and $g$ is strictly decreasing in $[1,q_1)$ and strictly increasing in $(q_1,2]$, with $g(\frac 4 3) = g(2) =1$. For $q \in (\frac 4 3,2)$, we have
$$g(q) \le 1 - M (\frac 1 q - \frac 1 2) (\frac 3 4 - \frac 1 q) \quad , \quad  M= 0.86326 \ . $$
\end{lemma}

\begin{proof}
(a) Differentiation gives $f'(q) = \frac{f(q)}{q^2} (\Psi(\frac 1 q) + \Psi(2-\frac 1 q))$. Since $\Gamma$ is logarithmic convex,$\Psi$ is increasing. Hence $\Psi(\frac 1 q) \le \Psi(1) = - \gamma$ and $\Psi(2-\frac 1 q) \le \Psi(\frac 3 2) = 2 (1- \ln(2)) - \gamma$, $\Psi(\frac 1 q)+\Psi(2-\frac 1 q) \le -2(\gamma + \ln(2) -1) < 0$. Therefore $f$ is decreasing in $[1,2]$. Moreover, $f(\frac 4 3) \le 0.7397$. \\

(b) For $g$ we find $(\ln(g))'(q)  = \frac 1 {2 q^2} (\Psi(2-\frac 1 q) - \Psi(\frac 1 q) - 2 \ln(2))$. The function $h(q):=\Psi(2-\frac 1 q) - \Psi(\frac 1 q) - 2 \ln(2)$ is strictly increasing, since with $\frac 1 q + \frac 1 p =1$ we have, using \eqref{eq3.1},
$$h'(q) = \frac 1 {q^2} (\Psi'(2-\frac 1 q) + \Psi'(\frac 1 q)) = \frac 1 {q^2} \sum_{n=1}^\infty (\frac 1 {(n+\frac 1 p)^2} + \frac 1 {(n-\frac 1 p)^2}) > 0 \ .$$
Moreover $h(1) = - 2 \ \ln(2) < 0$, $h(2) = 2 (1 - \ln(2)) > 0$. Thus $h$ has exactly one zero $q_1 \in (1,2)$, $q_1 \simeq 1.612$. We get that $g$ is decreasing in $[1,q_1)$ and increasing in $(q_1,2]$. We have $g(1) = \frac 2 {\sqrt \pi} > 1$, $g(\frac 4 3) = g(2) = 1$ and $g(q) < 1$ for $q \in (\frac 4 3 , 2)$. \\

For $\frac 4 3 < q < 2$, choose $\theta \in (0,1)$ with $\frac 1 q = (1-\theta) \frac 1 2 + \theta \frac 3 4$, $\theta = \frac 4 q -2$, $1-\theta = 3 - \frac 4 q$. Since $\Gamma$ is logarithmic convex, $F:= \ln \Gamma$ satisfies $F'' = \Psi' > 0$ and for some $\eta \in (\frac 1 2, \frac 3 4)$
$$F(\frac 1 q) \le (1-\theta) F(\frac 1 2) + \theta F(\frac 3 4) - \frac{\Psi'(\eta)} 2 (\frac 1 q - \frac 1 2) (\frac 3 4 - \frac 1 q) \ . $$
Since by \eqref{eq3.1} $\Psi'$ is decreasing, $\min_{\eta \in [\frac 1 2, \frac 3 4]} \Psi'(\eta) = \Psi'(\frac 3 4)$ and
$$F(\frac 1 q) \le (3 - \frac 4 q) F(\frac 1 2) + (\frac 4 q -2) F(\frac 3 4) - \frac{\Psi'(\frac 3 4)} 2 (\frac 1 q - \frac 1 2) (\frac 3 4 - \frac 1 q) \ . $$
Similarly, for $\frac 5 4 < 2 - \frac 1 q < \frac 3 2$, choose $\theta \in (0,1)$ with $2 - \frac 1 q = (1-\theta) {\frac 5 4} + \theta {\frac 3 2}$, $\theta = 3 - \frac 4 q$, $1 - \theta = \frac 4 q -2$, such that with  $\min_{\eta \in [\frac 5 4, \frac 3 2]} \Psi'(\eta) = \Psi'(\frac 3 2)$
$$F(2 - \frac 1 q) \le (\frac 4 q-2) F(\frac 5 4) + (3 - \frac 4 q) F(\frac 3 2) - \frac{\Psi'(\frac 3 2)} 2 (\frac 1 q - \frac 1 2) (\frac 3 4 - \frac 1 q) \ . $$
This yields after exponentiation with $c:=\frac{\Psi'(\frac 3 4)+\Psi'(\frac 3 2)} 4$
$$\Gamma(\frac 1 q) \Gamma(2-\frac 1 q) \le (\Gamma(\frac 1 2) \Gamma(\frac 3 2))^{3-\frac 4 q} (\Gamma(\frac 3 4) \Gamma(\frac 5 4))^{\frac 4 q -2}
\exp(- 2 c (\frac 1 q - \frac 1 2) (\frac 3 4 - \frac 1 q)) \ . $$
Clearly $\Gamma(\frac 1 2) \Gamma(\frac 3 2) = \frac \pi 2$, and by the complement formula for the $\Gamma$-function, cf. Abramowitz, Stegun \cite{AS},
$\Gamma(\frac 3 4) \Gamma(\frac 5 4) = \frac \pi {4 \sin(\frac \pi 4)} = \frac \pi {2 \sqrt 2}$, so that
$$g(q) = \sqrt{\frac{2^{\frac 2 q}}\pi \Gamma(\frac 1 q) \Gamma(2-\frac 1 q)} \le \exp(- c (\frac 1 q - \frac 1 2) (\frac 3 4 - \frac 1 q)) =: k(q) \ . $$
Let $\varepsilon := (\frac 1 q - \frac 1 2) (\frac 3 4 - \frac 1 q)$. Then $\varepsilon \le \frac 1 {64}$. Numerical evaluation yields $c \ge 0.86917$. By Taylor expansion
$k(q) \le 1 - c \varepsilon + \frac{c^2 \varepsilon^2} 2 = 1 - c (1-\frac{c \varepsilon} 2) \varepsilon \le 1 - d \varepsilon$, $d := c - \frac{c^2}{128} \ge 0.86326$ =: M. Therefore $g(q) \le 1 - M (\frac 1 q - \frac 1 2) (\frac 3 4 - \frac 1 q)$.
\end{proof}

\begin{lemma}\label{lem4}
For all $\frac 4 3 \le q \le 2$ and $0 \le s \le \frac {16} 5$, the functions $\delta_q$ of \eqref{eq2.6} satisfy $\delta_{\frac 4 3}(s) \le \delta_q(s) \le \delta_2(s) = \exp(-\frac{s^2} 4)$. Further $\delta_{\frac 4 3}(\frac{48}{25}) > 0$ and $\delta_{\frac 4 3}(\frac{16} 5) > - 0.588$.
\end{lemma}

\begin{proof}
Let $2 \le p = \frac q {q-1} \le 4$ be the conjugate index of $p$ and $s \in [\frac{48}{25},\frac{16} 5]$. We will show that $\frac d {dq} \delta_q(s) > 0$, or equivalently $\frac d {dp} \delta_q(s) < 0$.  For $m > -1$ we have
\begin{equation}\label{eq5.1}
\int_0^\infty r^m \exp(-r^p) dr = \frac 1 p \int_0^\infty u^{\frac{m+1} p -1} \exp(-u) du = \frac 1 {m+1} \Gamma(1+\frac{m+1} p) \ .
\end{equation}
Expanding $\cos(s r)$ into its Taylor series at zero, we find using \eqref{eq5.1}
\begin{equation}\label{eq5.2}
\delta_q(s) = \sum_{n=0}^\infty ( f_{2n}(p) \frac{s^{4n}}{(4n)!} - f_{2n+1}(p) \frac{s^{4n+2}}{(4n+2)!} )=: \sum_{n=0}^\infty F_n(p,s) \ ,
\end{equation}
where $f_{2n}(p) := \frac{\Gamma(1+\frac{4 n -1} p)}{\Gamma(1-\frac 1 p)}$, $f_{2n+1}(p) := \frac{\Gamma(1+\frac{4 n + 1} p)}{\Gamma(1-\frac 1 p)}$. Since $\Gamma$ is logarithmic convex, we have for  $x>0$ and $0 \le \theta \le1$ that
$\Gamma(x+\theta) \le \Gamma(x)^{1-\theta} \Gamma(x+1)^\theta = x^\theta \Gamma(x)$. Let $n \ge 2$, $x := 1 + \frac{4n-1} p \ge 1$ and $\theta := \frac 2 p$. We claim that $x^\theta = (1+ \frac{4n-1} p)^\theta \le \frac{4n+1} p$. This is equivalent to $(4n+p-1)^{\frac 2 p} p^{1-\frac 2 p} \le 4n+1$. Applying the inequality $a b \le \frac{a^r} r + \frac{b^{r'}}{r'}$ with $r := \frac p 2$ and $r' = \frac p {p-2}$, we get
$(4n+p-1)^{\frac 2 p} p^{1-\frac 2 p} \le \frac 2 p (4n+p-1) + p - 2 = \frac 2 p (4n-1) + p$ which is $\le 4n+1$ if and only if $n \ge \frac{p+1} 4$, which is satisfied, since $p \le 4$ and $n \ge 2$. Therefore
\begin{align*}
F_n(p,s) & = f_{2n}(p) \frac{s^{4n}}{(4n)!} ( 1 - \frac{f_{2n+1}(p)} {f_{2n}(p)} \frac{s^2}{(4n+1)(4n+2)} ) \\
& \ge f_{2n}(p) \frac{s^{4n}}{(4n)!} ( 1 - \frac{s^2}{2 p (2n+1)} ) > 0
\end{align*}
for $n \ge 2$ and $s \le \frac{16} 5 < \sqrt{20}$. Hence for all $m \ge 1$, $\delta_q(s) \ge \sum_{n=0}^m F_n(p,s)$. In particular, for $q=\frac 4 3$, we find by numerical evaluation $\delta_{\frac 4 3} (\frac{48}{25}) > 0.0026 > 0$, choosing $m=2$, and $\delta_{\frac 4 3} (\frac{16} 5) > - 0.588$, choosing $m=4$.
Formula \eqref{eq5.2} implies
\begin{align}\label{eq5.3}
\frac d {dp} \delta_q(s) & = \sum_{n=0}^\infty \frac d {dp} F_n(p,s) =:  - \frac 1 {p^2} \sum_{n=0}^\infty G_n(p,s)  \nonumber \\
& =: - \frac 1 {p^2} \sum_{n=0}^\infty (f_{2n}(p)g_{2n}(p) \frac{s^{4n}}{(4n)!} - f_{2n+1}(p)g_{2n+1}(p) \frac{s^{4n+2}}{(4n+2)!}) \ ,
\end{align}
where in terms of the Digamma function $\Psi$,
$g_{2n}(p)=  (4 n -1) \Psi(1+\frac{4n-1} p) + \Psi(1-\frac 1 p)$, $g_{2n+1}(p)=  (4 n + 1) \Psi(1+\frac{4n+1} p) + \Psi(1-\frac 1 p)$. By Abramowitz, Stegun \cite{AS}, 6.3, $\Psi'$ is positive, decreasing and $\Psi'(1+x) \le \frac 1 {x+\frac 1 2}$. Therefore
$$\Psi(1+\frac{4n+1} p) \le \Psi(1+\frac{4n-1} p) + \frac 2 p \Psi'(1+\frac{4n-1} p) \le \Psi(1+\frac{4n-1} p) + \frac 1 {2n} \ , $$
$$(4n+1) \Psi(1+\frac{4n+1} p) \le (4n-1) \Psi(1+\frac{4n-1} p) + 2 \Psi(1+\frac{4n-1} p) + \frac{4n+1}{2n} \ , $$
so that
\begin{equation}\label{eq5.4}\frac{g_{2n+1}(p)}{g_{2n}(p)} \le 1 + \frac{2 \Psi(1+\frac{4n-1} p) + 2 + \frac 1 {2n}}{g_{2n}(p)} = 1 + \frac 2 {4n-1} +
\frac{2 - \frac{2 \Psi(1-\frac 1 p)}{4n-1} + \frac 1 {2n}}{g_{2n}(p)} \ .
\end{equation}
We have $\Gamma(1-\frac 1 p) \in [-1.97,-1.08]$ and $\Psi(1+x) \ge \ln(1+x) - \frac 1 5$ for all $x \ge \frac 7 4$. Using this and \eqref{eq5.4}, calculation yields for $n \ge 2$
\begin{align*}
\frac{g_{2n+1}(p)}{g_{2n}(p)} & \le 1 + \frac 2 {4n-1} + c_n(p) \; , \; c_2(p) = \frac 1 6 + \frac p {10}, \\
c_3(p) & = \frac 1 {11} + \frac p {36}, \ c_4(p) = \frac 1 {18} + \frac p {72}, \ c_n(p) = \frac 2 {3 n \ln(n)} \text{ for } n \ge 5 \ .
\end{align*}
Then for $n \ge 5$, $\frac 2 {4n-1} + c_n(2) \le \frac 1 {n + \frac 1 4}$ and
\begin{align*}
\frac {f_{2n+1}(p) g_{2n+1}(p)} {f_{2n}(p) g_{2n}(p)} & \le \frac {4n+1} p (1+\frac 2 {4n-1} + c_n(p)) \le \frac {4n+1} 2 (1+\frac 2 {4n-1} + c_n(2))  \\
& =: q_n \le \frac{4n+1} 2 (1 + \frac 1 {n+ \frac 1 4}) = \frac{4n+5} 2 \ ,
\end{align*}
whereas $q_2 \le \frac{15} 2$, $q_3 \le \frac{25} 3$ and $q_4 \le \frac{21} 2$. This implies
$$G_n(p,s) \ge f_{2n}(p) g_{2n}(p) \frac{s^{4n}}{(4n)!} (1 - \frac{q_n s^2}{(4n+1)(4n+2)}) =: \tilde{G}_n(p,s) > 0 $$
for all $n \ge 2$ and $s \le \frac{16} 5 < \sqrt{12}$. Hence by \eqref{eq5.3} for all $m \ge 1$
$$\frac d {dp} \delta_q(s) \le - \frac 1 {p^2} \sum_{n=0}^m G_n(p,s) \ . $$
For $m=1$, with $g_0(s) = 0$, we have
$$\frac d {dp} \delta_q(s) \le + \frac {s^2} {2 p^2} [a(p) - b(p) \frac{s^2}{12} + c(p) \frac{s^4}{360} ] =: \phi(p,s) $$
with $-0.972 \le a(p) := f_1(p)g_1(p) \le -0.954$ varying very little, $-0.255 \le b(p) := f_2(p)g_2(p) \le 0.114$, $0 < b(p)$ for $p \le 2.83$ and
$c(p) := f_3(p) g_3(p)$ decreasing in $p \in [2,4]$, with value $6.66$ at $p=2$ and $1.64$ at $p=4$. Therefore $360 |a(p)| \ge 343.5$ and $\phi(p,s) < 0$ will be satisfied if
$$s^2 < 15 \frac{b(p)}{c(p)} + \sqrt{ (15 \frac{b(p)}{c(p)})^2 + \frac{343.5}{c(p)} } \ . $$
This holds for all $0 \le s \le \frac {16} 5$, if $c(p) \le 3.275$, i.e. $p \ge 2.81$. For $0 \le p \le 2$, the right side is minimal for $p=2$ and we require $s \le 2.72$. If $p < 2.81$ and $s > 2.72$ one needs two more terms, $m=3$, to show  $\frac d {dp} \delta_q(s) < 0$,
$$\frac d {dp} \delta_q(s) \le - \frac 1 {p^2} [ G_0(p,s) + G_1(p,s) + \tilde{G}_2(p,s) + \tilde{G}_3(p,s) ] < 0 \ . $$
\end{proof}

\begin{corollary}\label{cor}
For all $\frac 4 3 \le q \le 2$ and $\frac {48}{25} \le s \le \frac {16} 5$, $|\delta_q(s)| \le 0.588$.
\end{corollary}

\begin{proof}
By Lemma \ref{lem4}, $\delta_{\frac 4 3}(s) \le \delta_q(s) \le \delta_2(s) = \exp(-\frac{s^2} 4) \le \exp(-(\frac{24}{25})^2) < \frac 2 5$ for all $s \in [\frac{48}{25},\frac{16}5]$. By \eqref{eq2.7}, $\delta'_{\frac 4 3}(s) = - \frac{\Gamma(\frac 5 4)}{\Gamma(\frac 3 4)} s \gamma_4(s)$. According to Boyd \cite{Bo}, $\gamma_4(s) > 0$ for all $0 \le s \le 3.45$, the first positive zero of $\gamma_4$ being at $s_1 \simeq 3.4535$. Therefore $\delta_{\frac 4 3}$ is strictly decreasing in $[\frac{48}{25},\frac{16}5]$, with $\frac 2 5 > \delta_{\frac 4 3}(s) \ge \delta_{\frac 4 3}(\frac{16} 5) > - 0.588$ by Lemma \ref{lem4}. We conclude that
$$\max \{ |\delta_q(s)| \ | \ q \in [\frac 4 3, 2] \ , \ s \in [\frac{48}{25},\frac{16}5] \} \le 0.588 \ . $$
\end{proof}

{\bf Remark}. In fact, $\delta_q(s)$ is increasing in $q \in (1,2]$ and $s \in [0,\frac{24} 5]$. For $q \searrow 1$, $\delta_q(s) \to \cos(s)$, so that $|\delta_q(\pi)| \to 1$. \\

\section{Proof of Theorem \ref{th2}}

{\bf Proof of Theorem \ref{th1}. } \\
Barthe and Naor \cite{BN} showed for $1 \le q < \infty$ that
$$\lim_{n \to \infty} \frac{P_{n,q}(a^{(n)})}{P_{n,q}(a^{(2)})} = \sqrt{\frac{2^{\frac 2 q}} \pi \Gamma(\frac 1 q) \Gamma(2-\frac 1 q)} \ , $$
and this is $< 1$ if and only if $\frac 4 3 < q < 2$, cf. Lemma \ref{lem3} (b). Now consider $\frac 4 3 < q < 2$ and let $p = \frac q {q-1}$ be the conjugate index, $2 < p < 4$. As in the proof of Theorem \ref{th1}, we use $\cos(x) \ge 1 - \frac{x^2} 2 + \frac{x^4}{26} \>0$ for all $0 \le x \le \frac 3 2$, so that by \eqref{eq2.6} for all $s \le \frac 3 2 \sqrt n$
\begin{align*}
\delta_q(\frac s {\sqrt n}) & \ge \frac p {\Gamma(\frac 1 q)} \Big[ \int_0^{\frac 3 2 \frac{\sqrt n} s} (1 - \frac {s^2 r^2} {2n} + \frac {s^4 r^4} {26 n^2}) \ r^{p-2} \exp(-r^p) dr \\
& \quad \quad + \int_{\frac 3 2 \frac{\sqrt n} s}^\infty \cos(\frac{sr}{\sqrt n}) \ r^{p-2} \exp(-r^p) dr \Big] \\
& \ge \frac p {\Gamma(\frac 1 q)} \Big[ \int_0^\infty (1 - \frac {s^2 r^2} {2n} + \frac {s^4 r^4} {26 n^2}) \ r^{p-2} \exp(-r^p) dr - R \Big] \\
& = \frac 1 {\Gamma(1 - \frac 1 p)} \Big[ \Gamma(1 - \frac 1 p) - \Gamma(1 + \frac 1 p) \frac{s^2}{2 n} + \Gamma(1 + \frac 3 p) \frac{s^4}{26 n^2} - R \Big] \ ,
\end{align*}
\begin{align*}
R  & := \int_{\frac 3 2 \frac{\sqrt n} s}^\infty (2-\frac{s^2}{2 n} r^2 + \frac{s^4}{26 n^2} r^4) \ r^{p-2} \exp(-r^p) dr  \\
& = \frac 1 p \int_{(\frac 3 2 \frac{\sqrt n} s)^p}^\infty (2 u^{-\frac 1 p} - \frac{s^2}{2 n} u^{\frac 1 p} + \frac{s^4}{26 n^2} u^{\frac 3 p}) \exp(-u) du \ . \end{align*}
Then $u^{-\frac 1 p} \le \frac 2 3 \frac s {\sqrt n}$ and $u^{\frac 3 p} \le u^{\frac 3 2}$. As in the proof of Theorem 1 (iii), for $x \ge \frac 9 2$
$$\int_x^\infty u^{\frac 3 2} \exp(-u) du \le \Big( (1+x) (2+2 x +x^2) \Big)^{\frac 1 2} \exp(-x) \le \frac{13}{20} x^2 \exp(-x) \ . $$
Choose again $s \le \sqrt{\frac n 2}$. Then with $x := (\frac 3 2 \frac {\sqrt n} s)^p \ge (\frac 3 {\sqrt 2})^2 = \frac 9 2$ and $y := \frac s {\sqrt n}$
$$R \le \frac 1 p ( 2 (\frac 2 3 y) + \frac{y^4}{26} \frac{13}{20} (\frac 3 2 \frac 1 y)^{2p}) \exp(-(\frac 3 2 \frac 1 y)^p) \ . $$
We want $R \le \Gamma(1+\frac 3 p) \frac{y^4}{26}$, a condition which is strongest for $p=2$ when it means
$$\frac 1 2 ( \frac 4 3 y + \frac{81}{640}) < \frac{\Gamma(\frac 5 2)} {26} y^4 \exp(\frac 9 4 \frac 1 {y^2}) \ , $$
which is the same requirement as in (iii) of the proof of Theorem \ref{th1}, being valid for $0 \le y \le 0.7161$. Thus the choice of $s \le \sqrt{\frac n 2}$ is allowed and then
$$\delta_q(\frac s {\sqrt n}) \ge 1 - c \frac{s^2} n \quad , \quad c := \frac 1 2 \frac{\Gamma(1+\frac 1 p)}{\Gamma(1-\frac 1 p)} = \frac 1 2 \frac{\Gamma(2-\frac 1 q)}{\Gamma(\frac 1 q)} \ . $$
We have by Lemma \ref{lem3} (a) $\frac 1 4 \le c \le 0.3699 < \frac{37}{100}$, the lower estimate being attained for $p=q=2$, the upper valid for $p=4$, $q= \frac 4 3$. Therefore for $x := c \frac{s^2} n \le \frac{37}{200}$, $\ln(1-x) \ge -x - \frac 1 2 \frac{x^2}{1-x} \ge -x - \frac{100}{163} x^2$ and
$(1-c \frac{s^2} n)^n \ge \exp(-c s^2) (1 - \frac{100}{163} c^2 \frac{s^4} n)$. This yields the estimate
\begin{align*}
& \int_0^{\sqrt{\frac n 2}} \frac{1-\delta_q(\frac s {\sqrt n})^n}{s^2} ds \le \int_0^{\sqrt{\frac n 2}} \frac{1-\exp(-c s^2)(1-\frac{100}{163} c^2 \frac{s^4} n)}{s^2} ds \\
& \le \int_0^\infty \frac{1-\exp(-c s^2)}{s^2} ds - \int_{\sqrt{\frac n 2}}^\infty \frac{1-\exp(-c s^2)}{s^2} ds + \frac{100}{163} \frac{c^2} n \int_0^\infty s^2 \exp(-c s^2) ds \\
& = \sqrt{\pi c} (1+ \frac{25}{163} \frac 1 n) - \int_{\sqrt{\frac n 2}}^\infty \frac{1-\exp(-c s^2)}{s^2} ds \ .
\end{align*}
Therefore, using \eqref{eq2.6},
\begin{align*}
P_{n,q}(a^{(n)}) & = \Gamma(\frac 1 q) \frac 2 \pi \int_0^\infty \frac{1-\delta_q(\frac s {\sqrt n})^n}{s^2} ds \\
& = \Gamma(\frac 1 q) \frac 2 \pi \Big( \int_0^{\sqrt{\frac n 2}} \frac{1-\delta_q(\frac s {\sqrt n})^n}{s^2} ds + \int_{\sqrt{\frac n 2}}^\infty \frac{1-\delta_q(\frac s {\sqrt n})^n}{s^2} ds \Big) \\
& \le \Gamma(\frac 1 q) \frac 2 \pi \Big(\sqrt{\pi c} (1+ \frac{25}{163} \frac 1 n) + \int_{\sqrt{\frac n 2}}^\infty  \frac{\exp(-c s^2) - \delta_q(\frac s {\sqrt n})^n}{s^2} ds \Big) \\
& = \sqrt{\frac 2 \pi \Gamma(\frac 1 q) \Gamma(2-\frac 1 q)} (1+\frac{25}{163} \frac 1 n) + \Gamma(\frac 1 q) \frac 2 \pi S \ ,
\end{align*}
$S := \frac 1 {\sqrt n} \int_{\frac 1 {\sqrt 2}}^\infty  \frac{\exp(-c u^2 n) - \delta_q(u)^n}{u^2} du$. Since $c \ge \frac 1 4$ and $\delta_q(s) >0$ for all $0 \le u \le \frac {48}{25}$ by Lemma \ref{lem4}, we find
$$S \le   \frac 1 {\sqrt n} \int_{\frac 1 {\sqrt 2}}^\infty  \frac{\exp(- \frac{u^2} 4 n)}{u^2} du + \frac 1 {\sqrt n} \int_{\frac {48}{25}}^\infty \frac{|\delta_q(u)|^n}{u^2} du \ .$$
For $n \ge 8$, $v = \frac{u^2} 4 n \ge 1$ and
\begin{align*}
\frac 1 {\sqrt n} \int_{\frac 1 {\sqrt 2}}^\infty  \frac{\exp(- \frac{u^2} 4 n)}{u^2} du& = \frac 1 4 \int_{\frac n 8}^\infty \frac{\exp(-v)}{v^{\frac 3 2}} dv \\
& \le \frac 1 4 (\frac 8 n)^{\frac 3 2} \int_{\frac n 8}^\infty \exp(-v) dv = \frac{4 \sqrt 2}{n^{\frac 3 2}} \exp(-\frac n 8) \le \frac{4 \sqrt 2}{n^{\frac 3 2}} 0.8825^n \ .
\end{align*}
By Corollary \ref{cor} we have $|\delta_q(s)| \le 0.588$ for all $\frac {48}{25} \le u \le \frac {16} 5$. Therefore
$$\frac 1 {\sqrt n} \int_{\frac {48}{25}}^{\frac {16} 5} \frac{|\delta_q(u)|^n}{u^2} du \le \frac 1 {\sqrt n} 0.588^n \int_{\frac {48}{25}}^{\frac {16} 5} \frac {du}{u^2} = \frac 5 {24} \frac 1 {\sqrt n} 0.588^n \ . $$
Integration by parts shows for $\frac 4 3 \le q < 2$, $2 < p \le 4$ that
\begin{align*}
|\delta_q(u)| & = |\frac p {\Gamma(\frac 1 q)} \int_0^\infty \frac{\sin(ur)} u (p-2-p r^p) \ r^{p-3} \exp(-r^p) dr | \\
& \le \frac 1 u  \frac p {\Gamma(\frac 1 q)} \int_0^\infty |p-2-p r^p| \ r^{p-3} \exp(-r^p) dr \\
& = \frac 1 u \frac{2 p} {\Gamma(\frac 1 q)}(\frac{1-\frac 2 p} e)^{1-\frac 2 p} \le \frac{14} 5 \frac 1 u \ ,
\end{align*}
where the last inequality is the sharpest for $p=4$, $q= \frac 4 3$. Hence
$$\frac 1 {\sqrt n} \int_{\frac {16} 5}^\infty \frac{|\delta_q(u)|^n}{u^2} du \le \frac 1 {\sqrt n} (\frac{14} 5)^n \int_{\frac {16} 5}^\infty \frac{du}{u^{n+2}}
= \frac 5 {16} \frac 1 {\sqrt n (n+1)} (\frac 7 8)^n \le \frac 5 {16} \frac 1 {n^{\frac 3 2}} 0.875^n \ . $$
We finally get that
$$S \le \frac{4 \sqrt 2}{n^{\frac 3 2}} 0.8825^n + \frac 5 {24} \frac 1 {\sqrt n} 0.588^n + \frac 5 {16} \frac 1 {n^{\frac 3 2}} 0.875^n \ . $$
By Lemma \ref{lem3} (a) $\sqrt{\frac 2 \pi \frac{\Gamma(\frac 1 q)}{\Gamma(2 - \frac 1 q)}} \le \frac 2 {\sqrt \pi}$, with equality for $p=q=2$, and
$\frac 2 {\sqrt \pi} S \le \frac{0.10559} n$ for all $n > 20$. We conclude with $\frac {25}{163} + 0.10559 < 0.25896$ that
$$P_{n,q}(a^{(n)}) \le \sqrt{\frac 2 \pi \Gamma(\frac 1 q) \Gamma(2-\frac 1 q)} ( 1+ \frac{0.25896} n ) \ . $$
By Barthe, Naor \cite{BN} $P_{n,q}(a^{(2)}) = 2^{\frac 1 2 -\frac 1 q}$, so that together with Lemma \ref{lem4} (b)
\begin{align*}
\frac{P_{n,q}(a^{(n)})}{P_{n,q}(a^{(2)})} & \le \sqrt{\frac {2^{\frac 2 q}} \pi \Gamma(\frac 1 q) \Gamma(2-\frac 1 q)} ( 1+ \frac{0.25896} n ) \\
& \le \Big(1- M (\frac 1 q - \frac 1 2)(\frac 3 4 - \frac 1 q) \Big) \Big(1+ \frac{0.25896} n \Big) \quad , \quad M= 0.86326.
\end{align*}
Suppose that $n$ satisfies  $n \ge \frac 4 5 \frac{q^2}{(q-\frac 4 3)(2-q)}$. Then the last product is $< 1$, since $\frac{0.25896} n \le M (\frac 1 q - \frac 1 2)(\frac 3 4 - \frac 1 q)$ suffices and this requires $n > \frac{\frac 8 3 \frac{0.25896} M q^2}{(q-\frac 4 3)(2-q)}$, with $\frac 8 3 \frac{0.25896} M < \frac 4 5$. The condition  $n \ge \frac 4 5 \frac{q^2}{(q-\frac 4 3)(2-q)}$ will be satisfied if
$n > \frac{\frac{32}{15}}{q- \frac 4 3} + \frac{\frac{24} 5}{2-q} = \frac{\frac 8 3 q - \frac{32}{15}}{(q-\frac 4 3)(2-q)} \ge \frac{\frac 4 5 q^2}{(q-\frac 4 3)(2-q)}$, where the last inequality is an equality for $q=\frac 4 3$ and $q=2$. Hence for $n > \frac{\frac{32}{15}}{q- \frac 4 3} + \frac{\frac{24} 5}{2-q} =: \phi(q)$ we have $P_{n,q}(a^{(n)}) < P_{n,q}(a^{(2)})$. The restriction $n>20$ is automatically satisfied since the minimum of $\phi$ is $\phi(\frac 8 5) = 20$.
\hfill $\Box$  \\

As mentioned, $\lim_{n \to \infty} \frac{P_{n,q}(a^{(n)})}{P_{n,q}(a^{(2)})} = \sqrt{\frac {2^{\frac 2 q}} \pi \Gamma(\frac 1 q) \Gamma(2-\frac 1 q)}$. \\

\vspace*{1cm}

\noindent Mathematisches Seminar \\
Universit\"at Kiel \\
24098 Kiel, Germany \\
hkoenig@math.uni-kiel.de \\

\end{document}